\documentclass[11pt]{article}
\usepackage{amsmath}
\usepackage{amssymb}
\usepackage{amsthm}
\usepackage{mathrsfs}
\usepackage{geometry}
\usepackage{stmaryrd}
\usepackage{tikz}
\usepackage{xcolor}
\usepackage{appendix}
\usepackage{titlesec}
\usepackage{authblk}

\usepackage{url}
\usepackage[colorlinks,
            linkcolor=blue,
            urlcolor=green,
            citecolor=red]{hyperref}

\numberwithin{equation}{section}
\numberwithin{figure}{section}

\newtheorem{definition}{Definition} 
\newtheorem{theorem}{Theorem}
\newtheorem{lemma}[theorem]{Lemma} 
\newtheorem{corollary}[theorem]{Corollary}
\newtheorem{proposition}[theorem]{Proposition}

\newtheorem*{remark}{Remark}

\newcommand{\Address}{{
\bigskip
\footnotesize
\textsc{Jianghao Zhang: School of Mathematics, University of Edinburgh, King’s Buildings, Mayfield Road, Edinburgh EH9 3JZ, UK}\par\nopagebreak
\textit{Email address}: \texttt{1800010769@pku.edu.cn}
}}

\def\mathbi#1{\textbf{\em #1}}

\begin{document}
\title{\textbf{On Type \uppercase\expandafter{\romannumeral4} Superorthogonality}}
\author{Jianghao Zhang}
\date{}
\maketitle

\begin{abstract}
We prove the direct and the converse inequalities for type \uppercase\expandafter{\romannumeral4} superorthogonality in the vector-valued setting. The converse one is also new in the scalar setting.
\end{abstract}

\section{Introduction}
Fix a $\sigma\text{-finite}$ measure space $(X, \mu)$. Let $r\in \mathbb{N}_+$ and $\{f_l\}_{1\leq l\leq L}$ be a family of complex-valued functions in $L^{2r}(X, \mu)$. We are interested in the direct and the converse inequalities:
\begin{align}
\label{EquationSection1DirectInequalityScalar-Valued} & \|\sum_{l=1}^L f_l\|_{2r} \lesssim_{r} \|(\sum_{l=1}^{L} |f_l|^2)^{\frac{1}{2}}\|_{2r}, \\
\label{EquationSection1ConverseInequalityScalar-Valued} & \|(\sum_{l=1}^{L} |f_l|^2)^{\frac{1}{2}}\|_{2r} \lesssim_{r} \|\sum_{l=1}^L f_l\|_{2r}.
\end{align}
Inequalities like \eqref{EquationSection1DirectInequalityScalar-Valued} are a rigorous description of the cancellation property among the functions $\{f_l\}_{1\leq l\leq L}$. There are many examples of such inequalities in analysis, probability, and analytic number theory, including the Littlewood-Paley theory, the Burkholder–Davis–Gundy inequality for martingales, and certain exponential sum estimates. Inequalities of the form \eqref{EquationSection1ConverseInequalityScalar-Valued} are also proven to be useful. In particular, assuming additional conditions on $\{f_l\}_{1\leq l\leq L}$, we can combine \eqref{EquationSection1DirectInequalityScalar-Valued} and \eqref{EquationSection1ConverseInequalityScalar-Valued} to get Corollary \ref{CorollarySection1Littlewood-PaleyTheory} below.

To prove \eqref{EquationSection1DirectInequalityScalar-Valued}\eqref{EquationSection1ConverseInequalityScalar-Valued}, one way is expanding the $L^{2r}$ norms on both sides. Several terms only appear on one side of the expansions. It is natural to assume that $\{f_l\}_{1\leq l\leq L}$ obeys a kind of superorthogonality, which means many such terms vanish:
\begin{equation} \label{EquationSection1SuperorthogonalityScalar-Valued}
\int_{X} f_{l_1} \overline{f_{l_2}} \cdots f_{l_{2r-1}} \overline{f_{l_{2r}}} d\mu=0.
\end{equation}
This condition can be traced back to Khintchine's and Paley's works. See \cite{PierceOnSuperorthogonality} for some historical remarks. A crucial application of this type of condition (Type \uppercase\expandafter{\romannumeral2} superorthogonality defined below) then appeared in \cite{IWIonescu-Wainger}, where the boundedness of discrete singular Radon transforms was established. After such specific uses of superorthogonality, \cite{PierceOnSuperorthogonality} classified different kinds of superorthogonality and gave a comprehensive introduction on their development, variants, and applications. Now there are five types of superorthogonality:
\begin{description}
\item [Type \uppercase\expandafter{\romannumeral1}*]: \eqref{EquationSection1SuperorthogonalityScalar-Valued} holds whenever $l_1, l_3, \cdots, l_{2r-1}$ is not a permutation of $l_2, l_4, \cdots, l_{2r}$.
\item [Type \uppercase\expandafter{\romannumeral1}]: \eqref{EquationSection1SuperorthogonalityScalar-Valued} holds whenever some $l_j$ appears an odd number of times in $l_1, \cdots, l_{2r}$.
\item [Type \uppercase\expandafter{\romannumeral2}]: \eqref{EquationSection1SuperorthogonalityScalar-Valued} holds whenever some $l_j$ appears precisely once in $l_1, \cdots, l_{2r}$.
\item [Type \uppercase\expandafter{\romannumeral3}]: \eqref{EquationSection1SuperorthogonalityScalar-Valued} holds whenever some $l_j$ is strictly larger than all other indices in $l_1, \cdots, l_{2r}$.
\item [Type \uppercase\expandafter{\romannumeral4}]: \eqref{EquationSection1SuperorthogonalityScalar-Valued} holds whenever $l_1, \cdots, l_{2r}$ are all distinct.
\end{description}
Here the first four types were introduced in \cite{PierceOnSuperorthogonality}, and the last type was discoverd in  \cite{GPRYANewTypeOfSuperOrthogonality}. The assumption on $\{f_l\}_{1\leq l\leq L}$ is getting weaker from top to bottom. Notably, \cite{GPRYANewTypeOfSuperOrthogonality} showed that type \uppercase\expandafter{\romannumeral4} superorthogonality implies the direct inequality \eqref{EquationSection1DirectInequalityScalar-Valued}. For the converse inequality \eqref{EquationSection1ConverseInequalityScalar-Valued}, we see that a necessary condition is
\begin{equation}  \label{EquationSection1ConverseInequalityNecessaryConditionScalar-Valued}
\|(\sum_{l=1}^{L} |f_l|^{2r})^{\frac{1}{2r}}\|_{2r} \lesssim_{r} \|\sum_{l=1}^L f_l\|_{2r}.   
\end{equation}
In general, \eqref{EquationSection1ConverseInequalityNecessaryConditionScalar-Valued} is much weaker than \eqref{EquationSection1ConverseInequalityScalar-Valued}. It can be proved by interpolation in certain important cases. For example, assume $\{\Delta_l\}_l$ is defined as in Corollary \ref{CorollarySection1Littlewood-PaleyTheory} below and $f_l:=\Delta_lf$. One can usually check \eqref{EquationSection1ConverseInequalityNecessaryConditionScalar-Valued} with $r=1$ via orthogonality and with $r=\infty$ via a direct estimate. Interpolating these two estimates then gives \eqref{EquationSection1ConverseInequalityNecessaryConditionScalar-Valued} for any $1\leq r\leq \infty$. A classical work due to Paley \cite{PaleySuperorthogonality}, as explained in \cite{PierceOnSuperorthogonality}, showed that \eqref{EquationSection1ConverseInequalityNecessaryConditionScalar-Valued} is also sufficient for \eqref{EquationSection1ConverseInequalityScalar-Valued} if $\{f_l\}_{1\leq l\leq L}$ has type \uppercase\expandafter{\romannumeral3} superorthogonality.

One can also consider the superorthogonality phenomenon in the vector-valued setting. Assume $f_l, 1\leq l\leq L$, take values in a Hilbert space. We replace the condition \eqref{EquationSection1SuperorthogonalityScalar-Valued} with
\begin{equation*}
\int_{X} \langle f_{l_1}, f_{l_2} \rangle \cdots \langle f_{l_{2r-1}}, f_{l_{2r}} \rangle d\mu=0,  
\end{equation*}
and have the same classification of superorthogonality as above. In this vector-valued case, the relatively simple observation that type \uppercase\expandafter{\romannumeral2} superorthogonality implies the direct inequality \eqref{EquationSection1DirectInequalityScalar-Valued} was
used in \cite{MSZKIonescu-Wainger} and \cite{TaoIonescu-Wainger}.

For type \uppercase\expandafter{\romannumeral4} superorthogonality, we establish the vector-valued version of both direct and converse inequalities. The former extends its scalar version proved in \cite{GPRYANewTypeOfSuperOrthogonality}, while the latter is also new in the scalar setting. The main new ingredient is an algebraic identity for tensor products. Previous works such as \cite{PierceOnSuperorthogonality}\cite{GPRYANewTypeOfSuperOrthogonality} used scalar inequalities instead. Those inequalities are not sufficient to obtain the scalar converse inequality for type \uppercase\expandafter{\romannumeral4} superorthogonality. Our identity enables us to prove the direct and the scalar converse inequalities in a unified way. For the vector-valued converse inequality, we need to further exploit a cycle structure of the identity. See Lemma \ref{LemmaSection3Positivity}.

\begin{definition}
Let $V$ be a vector space over $\mathbb{R}$. A bilinear form $\mathfrak{B}: V^{2}\rightarrow \mathbb{C}$ is called positive if
\begin{equation*}
\mathfrak{B}(v, v)\geq 0, \forall v\in V.
\end{equation*}
In this case, denote $\mathfrak{B}(v, v)^{\frac{1}{2}}$ by $\mathfrak{B}(v)$.
\end{definition}
\begin{remark}
Note that every vector space over $\mathbb{C}$ can be regarded as a vector space over $\mathbb{R}$. In particular, the inner product on any complex Hilbert space is a positive bilinear form.
\end{remark}
\noindent Our first main result concerns the vector-valued direct inequality:
\begin{theorem} \label{TheoremSection1DirectInequality}
Fix a positive integer $r$, a $\sigma\text{-finite}$ measure space $(X, \mu)$, and a vector space $V$ over $\mathbb{R}$. Let $\mathfrak{B}$ be a positive bilinear form on $V$. Assume that the functions $f_l: X\rightarrow V, 1\leq l\leq L$, are such that $\mathfrak{B}(f_{l_{1}}, f_{l_{2}})\cdots \mathfrak{B}(f_{l_{2r-1}}, f_{l_{2r}})$ is integrable whenever $1\leq l_1,\cdots, l_{2r}\leq L$. Suppose $\{f_l\}_{1\leq l\leq L}$ satisfies type \uppercase\expandafter{\romannumeral4} superorthogonality:
\begin{equation} \label{EquationSection1Superorthogonality}
\int_{X} \mathfrak{B}(f_{l_{1}}, f_{l_{2}})\cdots \mathfrak{B}(f_{l_{2r-1}}, f_{l_{2r}}) d\mu=0,\ \text{whenever $l_1, \cdots, l_{2r}$ are all distinct}.
\end{equation}
Then
\begin{equation}  \label{EquationSection1DirectInequalityNewVersion}
\|\mathfrak{B}(\sum_{l=1}^L f_l)\|_{2r}\lesssim_{r} \|\big( \sum_{l=1}^L \mathfrak{B}(f_l)^2 \big)^{\frac{1}{2}}\|_{2r}.
\end{equation}
\end{theorem}
\noindent Our second main result is the vector-valued converse inequality:
\begin{theorem} \label{TheoremSection1ConverseInequality}
Under the same conditions of Theorem \ref{TheoremSection1DirectInequality}, if we further assume
\begin{equation} \label{EquationSection1ConverseInequalityNecessaryCondition}
\|\big( \sum_{l=1}^L \mathfrak{B}(f_l)^{2r} \big)^{\frac{1}{2r}}\|_{2r}\lesssim \|\mathfrak{B}(\sum_{l=1}^L f_l)\|_{2r},   
\end{equation}
then
\begin{equation*}
\|\big( \sum_{l=1}^L \mathfrak{B}(f_l)^2 \big)^{\frac{1}{2}}\|_{2r}\lesssim_{r} \|\mathfrak{B}(\sum_{l=1}^L f_l)\|_{2r}.
\end{equation*}
\end{theorem}
\noindent Combining Theorem \ref{TheoremSection1DirectInequality} and Theorem \ref{TheoremSection1ConverseInequality}, together with standard interpolation arguments, gives a new square function estimate. We only state the scalar case for simplicity.
\begin{corollary}  \label{CorollarySection1Littlewood-PaleyTheory}
Let $(X, \mu)$ be a $\sigma\text{-finite}$ measure space. Let $\Delta_l, 1\leq l\leq L$, be bounded self-adjoint operators on $L^2(X, \mu)$ such that $\sum\limits_{l=1}^L \Delta_l=I$. Assume $\Delta_l f\in L^p$ whenever $1\leq l\leq L$ and $1<p<\infty$, for any simple function $f$. Suppose for any $r\in \mathbb{N}_+$ and any simple function $f$, $\{\Delta_l f\}_{1\leq l\leq L}$ satisfies type \uppercase\expandafter{\romannumeral4} superorthogonality and
\begin{equation*}
\|(\sum_{l=1}^L |\Delta_{l} f|^{2r})^{\frac{1}{2r}}\|_{2r}\lesssim_r \|f\|_{2r}.
\end{equation*}
Then
\begin{equation*}
\|f\|_p \sim_p \|(\sum_{l=1}^L |\Delta_{l} f|^2)^{\frac{1}{2}}\|_p
\end{equation*}
whenever $1<p<\infty$ and $f\in L^p$.
\end{corollary}

A family of functions that has type \uppercase\expandafter{\romannumeral4} superorthogonality but doesn't have any stronger type was constructed in \cite{GPRYANewTypeOfSuperOrthogonality}. See Theorem 3 therein.

We end with remarks on superorthogonality in the context of the Littlewood-Paley theory. It
is not hard to see that martingale differences satisfy type \uppercase\expandafter{\romannumeral3} superorthogonality. Thus Paley's work on type \uppercase\expandafter{\romannumeral3} superorthogonality \cite{PaleySuperorthogonality}, developed in \cite{PierceOnSuperorthogonality}, essentially established the Littlewood-Paley theory for martingales. In contrast, type \uppercase\expandafter{\romannumeral3} superorthogonality can not
be applied to Euclidean Littlewood-Paley operators uniformly in the parameter $r$, while type \uppercase\expandafter{\romannumeral4} superorthogonality can be. More precisely, pick $\psi\in \mathcal{S}(\mathbb{R}^d)$ supported in $\{\xi: \frac{1}{2}\leq |\xi|\leq 2\}$ such that $\sum_{l\in \mathbb{Z}} \psi(\frac{\xi}{2^l})=1, \forall \xi\neq 0$. Define the Littlewood-Paley operators:
\begin{equation*}
\Delta_lf:=\big( \psi (\frac{\xi}{2^l})\hat{f}(\xi) \big)^{\vee}.
\end{equation*}
Divide $\{\Delta_lf\}_l$ into groups $\{\Delta_{10l+m}f\}_l, 0\leq m\leq 9$. Then for general $f$, each group has type \uppercase\expandafter{\romannumeral4} superorthogonality for any $r\in \mathbb{N}_+$ but does not have type \uppercase\expandafter{\romannumeral3} superorthogonality for large $r$.

\paragraph{Outline of the Paper.}
\begin{itemize} 
\item In Section \ref{Section2} we establish the algebraic identity Corollary \ref{CorollarySection2AlgebraicIdentity2} inspired by \cite{GPRYANewTypeOfSuperOrthogonality}.
\item In Section \ref{Section3} we use Corollary \ref{CorollarySection2AlgebraicIdentity2} to show Theorem \ref{TheoremSection1DirectInequality} and Theorem \ref{TheoremSection1ConverseInequality}.
\item In Section \ref{AppendixA} we discuss sharpness of the conditions in Theorem \ref{TheoremSection1DirectInequality} and Theorem \ref{TheoremSection1ConverseInequality}.
\end{itemize}

\paragraph{Notations and Conventions.}
\begin{itemize}
\item Given a parameter $r$ and two positive quantities $A, B$, we let $A\lesssim_r B$ denote the relation $A\leq C_rB$ for some constant $C_r$ depending on $r$. Write $A\sim_r B$ if $A\lesssim_r B$ and $B\lesssim_r A$.
\item $\left [n \right]:=\{1, \cdots, n\}, \forall n\in \mathbb{N}_{+}$.
\item Let $V_j, j\in \left[n \right]$ be vector spaces over a field $\mathbb{F}$. For each $j\in \left[n \right]$, we will work with $L$ vectors $v_{j, l_j}\in V_j, 1\leq l_j\leq L$. This $L$ is nonessential and often omitted. We also use the notation $S_j:=\sum\limits_{l_j} v_{j, l_j}=\sum\limits_{l_j=1}^L v_{j, l_j}$.
\item Given any $J\subset \left [n \right]$ and any $Z\subset \{l: 1\leq l\leq L\}$, the notation $\sum\limits_{(l_j)_{j\in J}\in Z^J}^*$ represents the sum over tuples $(l_j)_{j\in J}\in Z^J$ such that $l_j, j\in J$, are all distinct. We will omit $Z^J$ and use $\sum\limits_{(l_j)_{j\in J}}^*$ instead if $Z=\{l: 1\leq l\leq L\}$.
\item For a set $\mathscr{P}$ with elements denoted by $P$, $\sum\limits_{(l_P)_{P\in \mathscr{P}}}$ represents the sum over indices $1\leq l_P\leq L, P\in \mathscr{P}$, which are independent with each other.
\end{itemize}

\paragraph{Acknowledgement.} The author would like to thank Prof. Christoph Thiele for his numerous valuable writing suggestions. The author is also grateful to Yixuan Pang for his helpful comments and the anonymous referees for their detailed suggestions.


\section{A General Identity}  \label{Section2}
A key step in our argument is the identity proved below, because integrating it builds a relation among condition \eqref{EquationSection1Superorthogonality}, the left of \eqref{EquationSection1DirectInequalityNewVersion}, and the right of \eqref{EquationSection1DirectInequalityNewVersion}. We need the concept of partitions to state the identity.
\begin{definition}
A partition $\mathscr{P}$ of a set $A$ means $\mathscr{P}$ is a collection of non-empty disjoint sets and $A=\bigcup_{P\in \mathscr{P}} P$. For each $j\in A$, denote the unique $P\in \mathscr{P}$ containing $j$ by $\mathscr{P}(j)$.
\end{definition}
Our goal is to connect
\begin{equation*}
\sum\limits_{(l_j)_{j\in \left [n \right]}}^* v_{1, l_1}\otimes \cdots \otimes v_{n, l_n}.
\end{equation*}
with summations over independent indices.
\begin{theorem}  \label{TheoremSection2AlgebraicIdentity1}
Given $n\in \mathbb{N}_+$, there exist constants $C_{\mathscr{P}} \in \mathbb{Z}$ associated with each partition $\mathscr{P}$ of $\left [n \right]$ such that for $n$ arbitrary vector spaces $V_j, j\in \left [n \right]$, over a field $\mathbb{F}$ and $L$ vectors $v_{j, l_j}, 1\leq l_j\leq L$, in each $V_j$ we have
\begin{equation}  \label{EquationSection2AlgebraicIdentity1}
\sum\limits_{(l_j)_{j\in \left [n \right]}}^* v_{1, l_1}\otimes \cdots \otimes v_{n, l_n}=\sum_{\mathscr{P}} C_{\mathscr{P}} \sum_{(l_P)_{P\in \mathscr{P}}} v_{1, l_{\mathscr{P}(1)}}\otimes \cdots \otimes v_{n, l_{\mathscr{P}(n)}}.
\end{equation}
Moreover, the coefficients $C_{\mathscr{P}}$ satisfy the following:
\begin{enumerate}
\item $C_{\big{\{} \{1\}, \cdots, \{n\} \big{\}}}=1$.
\item If $n$ is even, $C_{\mathscr{P}}=(-1)^{\frac{n}{2}}$ whenever $\mathscr{P}$ is such that $\#P=2, \forall P\in \mathscr{P}$.
\end{enumerate}
\end{theorem}
\begin{remark}
For a vector $w$, $Cw, C\in \mathbb{Z}$, is defined via summation here.
\end{remark}

The following lemma is akin to the inclusion–exclusion principle.

\begin{lemma}  \label{LemmaSection2Inclusion–ExclusionPrinciple}
Let $A$ be a finite set, $W$ be a vector space, and $\phi, \psi: 2^A \rightarrow W$ be two arbitrary maps such that $\psi(\emptyset)=\phi(\emptyset)$. If for any $J\subset A$ we have
\begin{equation}  \label{EquationSection2Inclusion-ExclusionPrinciple}
\psi(J)=\phi(J)-\sum_{j\in J} \psi(J\setminus \{j\}),
\end{equation}
then
\begin{equation*}
\psi(A)=\sum_{J\subset A} (-1)^{\#J}(\#J)!\ \phi(J^c).
\end{equation*}
\end{lemma}
\begin{proof}
The proof is by induction on $\# A$. If $\# A=0$, the proposition follows from the assumption $\phi(\emptyset)=\psi(\emptyset)$. Now assume $\# A\geq 1$ and the proposition holds for any $A^{'}$ such that $\# A^{'}<\# A$. We show it also holds for $A$. Write
\begin{equation*}
\psi(A)=\phi(A)-\sum_{j\in A} \psi(A\setminus \{j\})
\end{equation*}
using \eqref{EquationSection2Inclusion-ExclusionPrinciple}.
For each $j$, apply the induction hypothesis to $A\setminus \{j\}$. We get
\begin{equation*}
\psi(A)=\phi(A)-\sum_{j\in A}\sum_{J^{'}\subset A\setminus \{j\}} (-1)^{\#J^{'}}(\#J^{'})!\ \phi \big((A\setminus \{j\})\setminus J^{'} \big).
\end{equation*}
Let $J=J^{'}\cup \{j\}$. Note that for each $J\subset A$, there are exactly $\# J$ pairs $(j, J^{'})$ satisfying $J=J^{'}\cup \{j\}$. Moreover, $\phi(A)$ corresponds to the term of $\emptyset$. This completes the proof.
\end{proof}

Below we show Theorem \ref{TheoremSection2AlgebraicIdentity1}.
\begin{proof}
We prove by induction on $n$. If $n=1$, both sides of \eqref{EquationSection2AlgebraicIdentity1} are
\begin{equation*}
\sum_{l} v_{1, l}.
\end{equation*}
For the $n=2$ case, we have
\begin{equation*}
\sum_{(l_j)_{j\in \left [2 \right]}}^* v_{1, l_1}\otimes v_{2, l_2}=\sum_{(l_j)_{j\in \left [2 \right]}} v_{1, l_1}\otimes v_{2, l_2}-\sum_{l} v_{1, l}\otimes v_{2, l}. 
\end{equation*}
as desired. Now assume $n\geq 3$ and the proposition holds for any $n^{'}$ such that $n^{'}<n$. We show it also holds for $n$. Let's first prove the existence of $C_{\mathscr{P}}$'s. Write
\begin{equation}  \label{EquationSection2Proof1}
\begin{split}
\sum\limits_{(l_j)_{j\in \left [n \right]}}^* v_{1, l_1}\otimes \cdots \otimes v_{n, l_n}=\sum\limits_{(l_j)_{j\in \left [n-1 \right]}}^* & v_{1, l_1}\otimes \cdots \otimes v_{n-1, l_{n-1}}\otimes S_n  \\
& -\sum_{j_0\in \left[n-1 \right]} \sum\limits_{(l_j)_{j\in \left [n-1 \right]}}^* v_{1, l_1}\otimes \cdots \otimes v_{n-1, l_{n-1}}\otimes v_{n, l_{j_0}}.
\end{split}
\end{equation}
For the first term on the right of \eqref{EquationSection2Proof1}, apply the induction hypothesis to
\begin{equation*}
\sum_{(l_j)_{j\in \left [n-1 \right]}}^* v_{1, l_1}\otimes \cdots \otimes v_{n-1, l_{n-1}},
\end{equation*}
which can be written as
\begin{equation}  \label{EquationSection2Proof2}
\sum_{\mathscr{P}^{'}} C_{\mathscr{P}^{'}}^{'} \sum_{(l_{P^{'}})_{P^{'}\in \mathscr{P}^{'}}} v_{1, l_{\mathscr{P}^{'}(1)}}\otimes \cdots \otimes v_{n-1, l_{\mathscr{P}^{'}(n-1)}}.
\end{equation}
Let $\mathscr{P}:=\mathscr{P}^{'}\cup \big{\{}\{n\} \big{\}}$ for each $\mathscr{P}^{'}$. We have
\begin{equation*}
\sum_{(l_{P^{'}})_{P^{'}\in \mathscr{P}^{'}}} v_{1, l_{\mathscr{P}^{'}(1)}}\otimes \cdots \otimes v_{n-1, l_{\mathscr{P}^{'}(n-1)}}\otimes S_n=\sum_{(l_P)_{P\in \mathscr{P}}} v_{1, l_{\mathscr{P}(1)}}\otimes \cdots \otimes v_{n, l_{\mathscr{P}(n)}}.
\end{equation*}
This is an allowed term of the right-hand side of \eqref{EquationSection2AlgebraicIdentity1} such that $\mathscr{P}(n)=\{n\}$.

Now consider the second term on the right of \eqref{EquationSection2Proof1}. For each $j_0\in \left [n-1 \right]$, we will show
\begin{equation*}
\sum\limits_{(l_j)_{j\in \left [n-1 \right]}}^* v_{1, l_1}\otimes \cdots \otimes v_{n-1, l_{n-1}}\otimes v_{n, l_{j_0}}
\end{equation*}
can be expressed as the right-hand side of \eqref{EquationSection2AlgebraicIdentity1}. It suffices to tackle the case $j_0=n-1$ since $V_1\otimes \cdots \otimes V_{n-1} \cong V_{\sigma^{'}(1)}\otimes \cdots \otimes V_{\sigma^{'}(n-1)}$ for any permutation $\sigma^{'}$ on $\left [n-1 \right]$. Write
\begin{equation}  \label{EquationSection2Proof3}
\begin{split}
\sum\limits_{(l_j)_{j\in \left [n-1 \right]}}^* & v_{1, l_1}\otimes \cdots \otimes v_{n-1, l_{n-1}}\otimes v_{n, l_{n-1}}  \\
& =\sum_{l_{n-1}}\big( \sum_{(l_j)_{j\in \left [n-2 \right]}\in (\{l_{n-1}\}^c)^{\left[ n-2 \right]}}^* v_{1, l_1}\otimes \cdots \otimes v_{n-2, l_{n-2}}\big)\otimes (v_{n-1, l_{n-1}}\otimes v_{n, l_{n-1}}).
\end{split}
\end{equation}
We eliminate the dependence between $l_{n-1}$ and other indices by applying Lemma \ref{LemmaSection2Inclusion–ExclusionPrinciple}. Given $J\subset \left[ n-2\right], j\in \left[ n-2\right]$, define $\mathbi{l}_{J, j}: \{l: 1\leq l\leq L\}^2 \rightarrow \{l: 1\leq l\leq L\}$ as follows:
\begin{equation*}
\mathbi{l}_{J, j}(k, m):=\begin{cases}
k,\ j\in J, \\
m,\ j\in J^c.
\end{cases}
\end{equation*}
Take
\begin{equation*}
\phi(\emptyset)=\psi(\emptyset):=v_{1, l_{n-1}}\otimes \cdots \otimes v_{n-2, l_{n-1}},
\end{equation*}
and
\begin{align*}
\phi(J) & :=\sum_{(l_j)_{j\in J}}^* v_{1, \mathbi{l}_{J, 1}(l_1, l_{n-1})}\otimes \cdots \otimes v_{n-2, \mathbi{l}_{J, n-2}(l_{n-2}, l_{n-1})},  \\
\psi(J) & :=\sum_{(l_j)_{j\in J}\in (\{l_{n-1}\}^c)^J}^* v_{1, \mathbi{l}_{J, 1}(l_1, l_{n-1})}\otimes \cdots \otimes v_{n-2, \mathbi{l}_{J, n-2}(l_{n-2}, l_{n-1})},
\end{align*}
for $J\neq \emptyset$. Then $\phi, \psi: 2^{\left[ n-2 \right]}\rightarrow V_1\otimes \cdots \otimes V_{n-2}$, satisfy \eqref{EquationSection2Inclusion-ExclusionPrinciple}. Applying Lemma \ref{LemmaSection2Inclusion–ExclusionPrinciple} gives
\begin{equation}  \label{EquationSection2Proof4}
\begin{split}
\sum_{(l_j)_{j\in \left [n-2 \right]}\in (\{l_{n-1}\}^c)^{\left[ n-2 \right]}}^* & v_{1, l_1}\otimes \cdots \otimes v_{n-2, l_{n-2}}  \\
& =\sum_{J\subset \left[n-2 \right]} (-1)^{\#J}(\#J)!\ \sum_{(l_j)_{j\in J^c}}^* v_{1, \mathbi{l}_{J^c, 1}(l_1, l_{n-1})}\otimes \cdots \otimes v_{n-2, \mathbi{l}_{J^c, n-2}(l_{n-2}, l_{n-1})}.
\end{split}
\end{equation}
Since $V_1\otimes \cdots \otimes V_{n-2} \cong V_{\sigma^{''}(1)}\otimes \cdots \otimes V_{\sigma^{''}(n-2)}$ for any permutation $\sigma^{''}$ on $\left [n-2 \right]$, we can apply the induction hypothesis to
\begin{equation*}
\sum_{(l_j)_{j\in J^c}}^* (\bigotimes_{j\in J^c} v_{j, l_j})
\end{equation*}
and write it as
\begin{equation}  \label{EquationSection2Proof5}
\sum_{\mathscr{P}^{''}} C_{\mathscr{P}^{''}}^{''} \sum_{(l_{P^{''}})_{P^{''}\in \mathscr{P}^{''}}} (\bigotimes_{j\in J^c} v_{j, l_{\mathscr{P}^{''}(j)}}).
\end{equation}
Plug \eqref{EquationSection2Proof5} into \eqref{EquationSection2Proof4}, and then plug \eqref{EquationSection2Proof4} into \eqref{EquationSection2Proof3}. After changing the summation order, \eqref{EquationSection2Proof3} is expressed as
\begin{equation*}
\sum_{J\subset \left[n-2 \right]} (-1)^{\#J}(\#J)!\ \sum_{\mathscr{P}^{''}} C_{\mathscr{P}^{''}}^{''} \Box_{J, \mathscr{P}^{''}},
\end{equation*}
where $\Box_{J, \mathscr{P}^{''}}$ represents the following:
\begin{equation*}
\big( \sum_{l_{n-1}} \sum_{(l_{P^{''}})_{P^{''}\in \mathscr{P}^{''}}} v_{1, \mathbi{l}_{J^c, 1}(l_{\mathscr{P}^{''}(1)}, l_{n-1})}\otimes \cdots \otimes v_{n-2, \mathbi{l}_{J^c, n-2}(l_{\mathscr{P}^{''}(n-2)}, l_{n-1})} \big) \otimes (v_{n-1, l_{n-1}}\otimes v_{n, l_{n-1}}).
\end{equation*}
It suffices to show that $\Box_{J, \mathscr{P}^{''}}$ can be expressed as the right-hand side of \eqref{EquationSection2AlgebraicIdentity1} for each pair $(J, \mathscr{P}^{''})$. Let $\mathscr{P}:=\mathscr{P}^{''}\bigcup \big\{ J\cup \{n-1,n\} \big\}$. We see
\begin{equation}  \label{EquationSection2Proof6}
\Box_{J, \mathscr{P}^{''}}=\sum_{(l_P)_{P\in \mathscr{P}}}v_{1, l_{\mathscr{P}(1)}}\otimes \cdots \otimes v_{n, l_{\mathscr{P}(n)}}.
\end{equation}
So we obtain the existence of $C_{\mathscr{P}}$'s.

We can track $C_{\mathscr{P}}$'s as follows: For $\mathscr{P}=\big{\{} \{1\}, \cdots, \{n\} \big{\}}$, it only appears when expressing the first term on the right of \eqref{EquationSection2Proof1} since every $\mathscr{P}$ from the second term contains $\{j_0, n\}$ for some $j_0\in \left[ n-1\right]$. So we get $C_{\big{\{} \{1\}, \cdots, \{n\} \big{\}}}=C_{\big{\{} \{1\}, \cdots, \{n-1\} \big{\}}}^{'}=1$ in \eqref{EquationSection2Proof2}.

Now assume $n$ is even. For each $\mathscr{P}$ such that $\#P=2$ for all $P\in \mathscr{P}$, it only appears when expressing the second term on the right of \eqref{EquationSection2Proof1} since every $\mathscr{P}$ from the first term contains $\{n\}$. Moreover, there exists exactly one $j_0$ such that $\{j_0, n\}\in \mathscr{P}$. This implies $\mathscr{P}$ can only appear in the expression of
\begin{equation*}
\sum_{(l_j)_{j\in \left [n-1 \right]}}^* v_{1, l_1}\otimes \cdots \otimes v_{n-1, l_{n-1}}\otimes v_{n, l_{j_0}}.
\end{equation*}
According to our induction process, we can assume $j_0=n-1$ without loss of generality. Then the right-hand side of \eqref{EquationSection2Proof6} corresponds to $\mathscr{P}$ exactly when $(J, \mathscr{P}^{''})=(\emptyset, \mathscr{P}\backslash \big{\{} \{n-1, n\}\big{\}})$. Thus $C_{\mathscr{P}}=-C_{\mathscr{P}\backslash \big{\{} \{n-1, n\}\big{\}}}^{''}=(-1)^{\frac{n}{2}}$, which completes the proof.
\end{proof}

Equivalently, we have the identity for general multilinear maps.
\begin{corollary}  \label{CorollarySection2AlgebraicIdentity2}
Given $n\in \mathbb{N}_+$, let $V_j, j\in \left [n \right]$, be $n$ arbitrary vector spaces over a field $\mathbb{F}$. For each $j$, let $v_{j, l_j}, 1\leq l_j\leq L$, be $L$ vectors in $V_j$. Then for any vector space $W$ and any $n\text{-linear}$ map $\Lambda: V_1\times \cdots \times V_n\rightarrow W$, we have
\begin{equation}  \label{formidentity}
\sum_{(l_j)_{j\in \left [n \right]}}^* \Lambda(v_{1, l_1}, \cdots, v_{n, l_n})=\sum_{\mathscr{P}} C_{\mathscr{P}} \sum_{(l_P)_{P\in \mathscr{P}}} \Lambda(v_{1, l_{\mathscr{P}(1)}}, \cdots, v_{n, l_{\mathscr{P}(n)}}),
\end{equation}
where $C_{\mathscr{P}}$'s are the same constants in Theorem \ref{TheoremSection2AlgebraicIdentity1}.
\end{corollary}

The above expression of $\sum\limits_{(l_j)_{j\in \left [n \right]}}^*$ in terms of $\sum\limits_{(l_P)_{P\in \mathscr{P}}}$ is unique. We will not use this fact, but a short proof is included for completeness.
\begin{proposition}  \label{PropositionSection2Uniqueness}
The $C_{\mathscr{P}}$'s in Theorem \ref{TheoremSection2AlgebraicIdentity1} and Corollary \ref{CorollarySection2AlgebraicIdentity2} are uniquely determined.
\end{proposition}
\begin{proof}
Fix an expression \eqref{formidentity}. Take $\mathbb{F}=\mathbb{C}, V_j=\mathbb{C}, j\in \left [n \right], W=\mathbb{C}$, and $\Lambda(v_1, \cdots, v_n)=v_1\cdots v_n$. Let $v_{j, l}=z_{l}^{2^j}, j\in \left [n \right], 1\leq l\leq L$. Then
\begin{equation*}
\sum_{(l_P)_{P\in \mathscr{P}}} v_{1, l_{\mathscr{P}(1)}}\cdots v_{n, l_{\mathscr{P}(n)}}=\prod_{P\in \mathscr{P}}(\sum_{l} z_l^{k_P}),
\end{equation*}
where $k_P:=\sum\limits_{j\in P} 2^j$. Note that $k_P\neq k_{\tilde{P}}$ if $P\neq \tilde{P}$. Define
\begin{equation*}
S^{(k)}:=\sum_{l} z_{l}^k, k\in \mathbb{N}_{+}.
\end{equation*}
Then \eqref{formidentity} becomes
\begin{equation}  \label{EquationSection2Proof7}
\sum_{(l_j)_{j\in \left [n \right]}}^* z_{l_1}^{2^1}\cdots z_{l_n}^{2^n}=\sum_{\mathscr{P}} C_{\mathscr{P}} \prod_{P\in \mathscr{P}} S^{(k_P)}.
\end{equation}
Take $L=\sum\limits_{j\in \left [n \right]} 2^j$. Given $C^{(k)}\in \mathbb{C}, 1\leq k\leq L$, the system of equations:
\begin{equation}  \label{EquationSection2Proof8}
S^{(k)}=C^{(k)}, 1\leq k\leq L,  
\end{equation}
has a unique solution $(z_l)_{1\leq l\leq L}$. One can see this by using the inverse of Newton's identities. Now we can determine each $C_{\mathscr{P}}$ as follows: Fix $\mathscr{P}$. Let $C^{(k)}=1$ if $k=k_P$ for some $P\in \mathscr{P}$ and $C^{(k)}=0$ otherwise. Denote the corresponding solution of \eqref{EquationSection2Proof8} by $\big(z_l(\mathscr{P}) \big)_{1\leq l\leq L}$. Take $z_l=z_l(\mathscr{P})$ in \eqref{EquationSection2Proof7}. We obtain
\begin{equation}  \label{EquationSection2Proof9}
\sum_{(l_j)_{j\in \left [n \right]}}^* z_{l_1}(\mathscr{P})^{2^1}\cdots z_{l_n}(\mathscr{P})^{2^n}=C_{\mathscr{P}}.
\end{equation}
The left-hand side of \eqref{EquationSection2Proof9} only depends on $\mathscr{P}$, which implies the uniqueness of $C_{\mathscr{P}}$.
\end{proof}

For each $\mathscr{P}$ such that $\#P=2$ for all $P\in \mathscr{P}$, our proof of Theorem \ref{TheoremSection2AlgebraicIdentity1} determines $C_{\mathscr{P}}$ in a way that doesn't involve other types of $C_{\mathscr{P}}$'s. Moreover, the proofs of Theorem \ref{TheoremSection2AlgebraicIdentity1} and Proposition \ref{PropositionSection2Uniqueness} give two theoretically possible ways to find the values of all $C_{\mathscr{P}}$'s, but both are very complicated. A different induction scheme is given in \cite{pang2024shortproofidentityrelated}, where all $C_{\mathscr{P}}$'s are calculated explicitly.


\section{Sufficiency of Type \texorpdfstring{\uppercase\expandafter{\romannumeral4} Superorthogonality}{IV Superorthogonality}}  \label{Section3}
Given a vector space $V$ over $\mathbb{R}$ and a positive bilinear form $\mathfrak{B}$, we can replace $\mathfrak{B}$ with its symmetrization 
\begin{equation*}
\mathfrak{B}_{sym}(v_1, v_2):=\frac{1}{2}\big( \mathfrak{B}(v_1, v_2)+\mathfrak{B}(v_2, v_1) \big),\ \forall v_1, v_2\in V.
\end{equation*}
The conditions of Theorem \ref{TheoremSection1DirectInequality} and Theorem \ref{TheoremSection1ConverseInequality} still hold for $\mathfrak{B}_{sym}$, and the conclusions are invariant. Thus we can assume $\mathfrak{B}$ is symmetric from now on.

For any $v_1, v_2\in V$, expanding $\mathfrak{B}(v_1+tv_2, v_1+tv_2)\geq 0, \forall t\in \mathbb{R}$, gives $\mathfrak{B}(v_1, v_2)\in \mathbb{R}$ and the Cauchy–Schwarz inequality:
\begin{equation}  \label{EquationSection3Cauchy–Schwarz}
|\mathfrak{B}(v_1, v_2)|\leq \mathfrak{B}(v_1)\mathfrak{B}(v_2).
\end{equation}
From now on, we fix $\Lambda$ to be the following $2r\text{-linear}$ map:
\begin{equation}  \label{EquationSection3DefiningLambda}
\Lambda(v_1, \cdots, v_{2r}):=\mathfrak{B}(v_1, v_2)\cdots \mathfrak{B}(v_{2r-1}, v_{2r}),\ \forall v_1, \cdots, v_{2r}\in V,
\end{equation}
and use the notation:
\begin{equation*}
\Lambda(v):=\Lambda(v, \cdots, v)^{\frac{1}{2r}}=\mathfrak{B}(v, v)^{\frac{1}{2}}=\mathfrak{B}(v), \forall v\in V.
\end{equation*}
Note that \eqref{EquationSection3Cauchy–Schwarz} implies
\begin{equation}  \label{EquationSection3Log-Convexity}
|\Lambda(v_1, \cdots, v_{2r})|\leq \Lambda(v_1)\cdots \Lambda(v_{2r}),\ \forall v_1, \cdots, v_{2r}\in V.
\end{equation}

\subsection{Proof of Theorem \ref{TheoremSection1DirectInequality}}  \label{Subsection3.1}
Applying Corollary \ref{CorollarySection2AlgebraicIdentity2} with $n=2r, V_j=V, v_{j, l_j}=f_{l_j}, j\in \left[2r \right]$ and $\Lambda$ defined in \eqref{EquationSection3DefiningLambda} gives
\begin{equation}  \label{EquationSection3DirectInequality1}
\begin{split}
\sum_{(l_j)_{j\in \left [2r \right]}}^* \Lambda(f_{l_1}, \cdots, f_{l_{2r}}) & =\sum_{\mathscr{P}} C_{\mathscr{P}} \sum_{(l_P)_{P\in \mathscr{P}}} \Lambda(f_{l_{\mathscr{P}(1)}}, \cdots, f_{l_{\mathscr{P}(2r)}})  \\
& =\Lambda(\sum_{l} f_l)^{2r}+\sum_{\mathscr{P}\neq \big{\{} \{1\}, \cdots, \{2r\} \big{\}}} C_{\mathscr{P}} \sum_{(l_P)_{P\in \mathscr{P}}} \Lambda(f_{l_{\mathscr{P}(1)}}, \cdots, f_{l_{\mathscr{P}(2r)}}).  
\end{split}
\end{equation}
Integrate both sides of \eqref{EquationSection3DirectInequality1} over $X$. The vanishing condition \eqref{EquationSection1Superorthogonality} implies
\begin{equation}  \label{EquationSection3DirectInequality2}
\int_X \sum_{(l_j)_{j\in \left [2r \right]}}^* \Lambda(f_{l_1}, \cdots, f_{l_{2r}}) d\mu=0.  
\end{equation}
Hence the triangle inequality implies
\begin{equation}  \label{EquationSection3DirectInequality3}
\int_X \Lambda(\sum_{l} f_l)^{2r} d\mu \leq \sum_{\mathscr{P}\neq \big{\{} \{1\}, \cdots, \{2r\} \big{\}}} |C_{\mathscr{P}}| \int_X |\sum_{(l_P)_{P\in \mathscr{P}}} \Lambda(f_{l_{\mathscr{P}(1)}}, \cdots, f_{l_{\mathscr{P}(2r)}})| d\mu.    
\end{equation}

For any partition $\mathscr{P}$, we have
\begin{equation}  \label{EquationSection3DirectInequality4}
|\sum_{(l_P)_{P\in \mathscr{P}}} \Lambda(f_{l_{\mathscr{P}(1)}}, \cdots, f_{l_{\mathscr{P}(2r)}})|\leq \Lambda(\sum_{l} f_l)^{\# \{P: \#P=1\}} \prod_{P: \#P\geq 2} \big(\sum_l \Lambda(f_l)^{\#P} \big)
\end{equation}
by first summing over $l_P$ such that $\#P=1$, then applying \eqref{EquationSection3Log-Convexity} to each term, and finally summing over $(l_P)_{P: \#P\geq 2}$ and via the distributive law. Now we successively apply \eqref{EquationSection3DirectInequality4}, monotonicity of $l^p$ norms, and H\"older's inequality:
\begin{equation}  \label{EquationSection3DirectInequality5}
\begin{split}
\int_X |\sum_{(l_P)_{P\in \mathscr{P}}} & \Lambda(f_{l_{\mathscr{P}(1)}}, \cdots, f_{l_{\mathscr{P}(2r)}})| d\mu  \\
& \leq\int_X \Lambda(\sum_{l} f_l)^{\# \{P: \#P=1\}} \prod_{P: \#P\geq 2} \big(\sum_l \Lambda(f_l)^{\#P} \big) d\mu  \\
& \leq\int_X \Lambda(\sum_{l} f_l)^{\# \{P: \#P=1\}} \big(\sum_l \Lambda(f_{l})^{2} \big)^{\frac{2r-\# \{P: \#P=1\}}{2}} d\mu  \\
& \leq \Big(\int_X \Lambda(\sum_{l} f_l)^{2r} d\mu \Big)^{\frac{\# \{P: \#P=1\}}{2r}} \Big(\int_X \big(\sum_l \Lambda(f_{l})^{2} \big)^{r}  d\mu \Big)^{\frac{2r-\# \{P: \#P=1\}}{2r}}.
\end{split}
\end{equation}

Plug \eqref{EquationSection3DirectInequality5} into the right-hand side of \eqref{EquationSection3DirectInequality3}. Divide both sides by $\|\big( \sum\limits_l\Lambda(f_l)^2 \big)^{\frac{1}{2}}\|_{2r}^{2r}$ and let
\begin{equation*}
t:=\frac{\|\Lambda(\sum\limits_l f_l)\|_{2r}}{\|\big( \sum\limits_l\Lambda(f_l)^2 \big)^{\frac{1}{2}}\|_{2r}}.    
\end{equation*}
We obtain
\begin{equation*}
t^{2r}\leq Q(t)    
\end{equation*}
for some polynomial $Q$ such that $deg(Q)\leq 2r-2$ since $\mathscr{P}\neq \big{\{} \{1\}, \cdots, \{2r\} \big{\}}$ on the right of \eqref{EquationSection3DirectInequality3}. This implies $t\lesssim_r 1$ and completes the proof.


\subsection{Proof of Theorem \ref{TheoremSection1ConverseInequality}}  \label{Subsection3.2}
For the converse inequality, we need to explore the positivity of certain summations.
\begin{lemma}  \label{LemmaSection3Positivity}
Let $v_l, 1\leq l\leq L$ be vectors in $V$. Then for any partition $\mathscr{P}$ of $[2r]$ such that $\#P=2, \forall P\in \mathscr{P}$ we have
\begin{equation*}
\sum_{(l_P)_{P\in \mathscr{P}}} \Lambda(v_{l_{\mathscr{P}(1)}}, \cdots, v_{l_{\mathscr{P}(2r)}})\geq 0.
\end{equation*}
\end{lemma}
\begin{proof}
Fix such a partition $\mathscr{P}$. We can construct a graph $G$ as follows:
\begin{enumerate}
\item Each $P\in \mathscr{P}$ is a vertex of $G$.
\item For any unordered pair $P, P^{'}\in G$ (we allow $P=P^{'}$.), add an edge connecting them if there exists $j: 1\leq j\leq r$ such that $\mathscr{P}(2j-1)=P, \mathscr{P}(2j)=P^{'}$.
\end{enumerate}
Note that there might exist loops and multiple edges. Each vertex of $G$ has degree $2$, so $G$ is the union of disjoint cycles $\{\mathscr{C}_k\}_k$ (This can be proved by induction on the number of vertices). Let $a(k)$ denote the number of vertices in $\mathscr{C}_k$. We can then write
\begin{equation*}
\begin{split}
\sum_{(l_P)_{P\in \mathscr{P}}} \Lambda(v_{l_{\mathscr{P}(1)}}, \cdots, v_{l_{\mathscr{P}(2r)}})=\prod_{k} \big( \sum_{l_1, \cdots, l_{a(k)}}\prod_{j=1}^{a(k)} \mathfrak{B}(v_{l_j}, v_{l_{j+1}}) \big),
\end{split}
\end{equation*}
where $l_{a(k)+1}:=l_1$, and similar definitions apply to the following argument. It suffices to show
\begin{equation*}
\sum_{l_1, \cdots, l_{a}}\prod_{j=1}^{a} \mathfrak{B}(v_{l_j}, v_{l_{j+1}})\geq 0
\end{equation*}
for $\forall a\geq 1$.

If $a$ equals to $1$ or $2$, the positivity follows immediately from the expression. We only need to handle the $a\geq 3$ case. This is done via fixing the cycle's axis of symmetry and folding the two halves under other summations. See Figure \ref{Folding the circles}.
\begin{figure}  
\centering
\includegraphics[width=0.5\textwidth,height=0.3\textwidth]{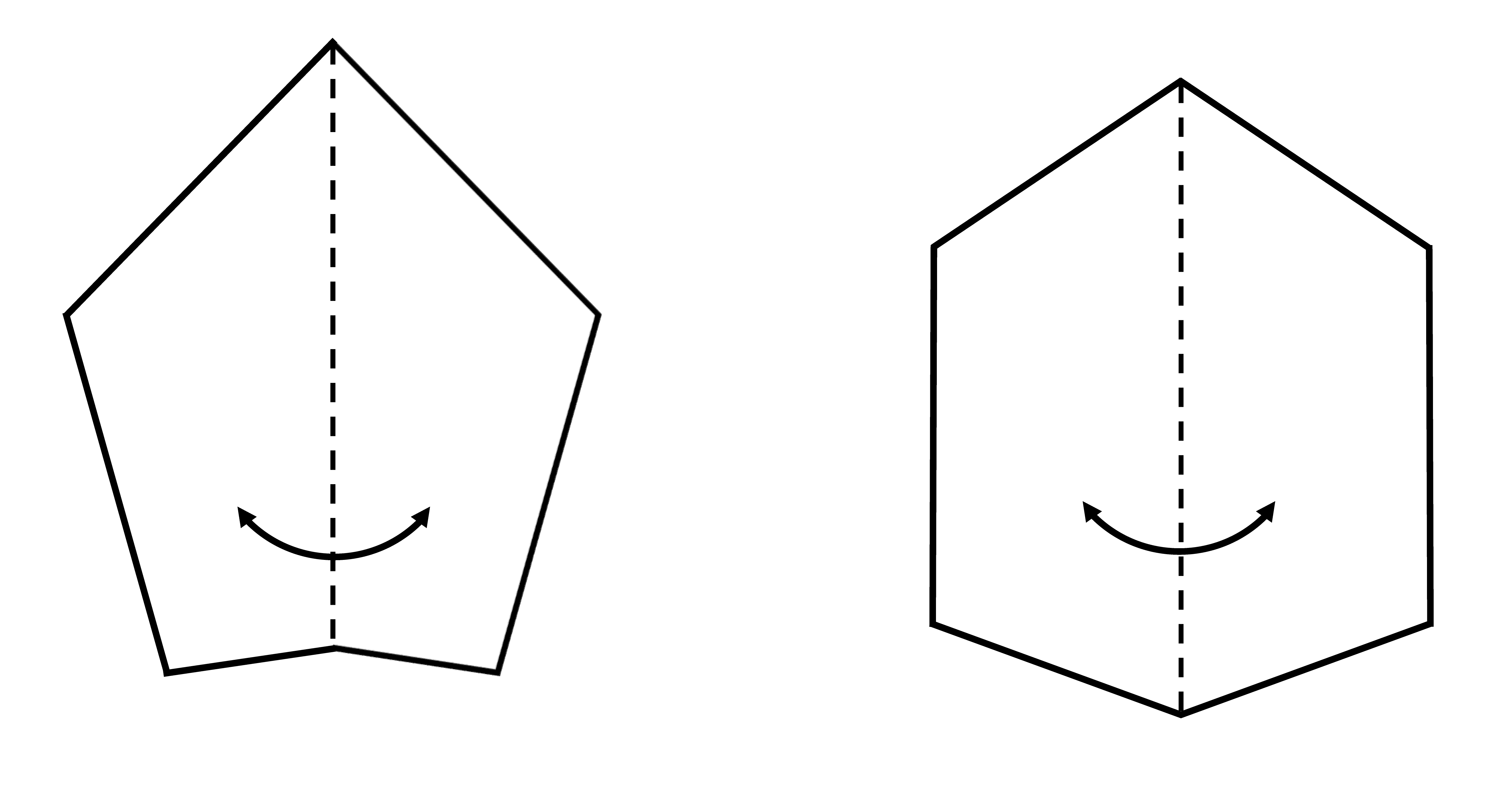}
\caption{Folding the cycles}
\label{Folding the circles}
\end{figure}
More precisely, for odd $a=2m-1$ we have
\begin{equation*}
\begin{split}
\sum_{l_1, \cdots, l_{2m-1}} & \prod_{j=1}^{2m-1} \mathfrak{B}(v_{l_j}, v_{l_{j+1}})  \\
& =\sum_{l_1} \sum_{l_2, \cdots, l_m}\sum_{l_{2m-1}, \cdots, l_{m+1}} \mathfrak{B}(v_{l_1}, v_{l_2})\cdots \mathfrak{B}(v_{l_m}, v_{l_{m+1}})\cdots \mathfrak{B}(v_{l_{2m-1}}, v_{l_1}) \\
& = \sum_{l_1} \mathfrak{B}\Big( \sum_{l_2, \cdots, l_m} \big(\prod_{j=1}^{m-1}\mathfrak{B}(v_{l_j}, v_{l_{j+1}}) \big)v_{l_m}, \sum_{l_2, \cdots, l_m} \big(\prod_{j=1}^{m-1}\mathfrak{B}(v_{l_j}, v_{l_{j+1}}) \big)v_{l_m} \Big) \\
& \geq 0.
\end{split}
\end{equation*}
Here we used the fact that $\mathfrak{B}$ is symmetric and positive. And for even $a=2m$ we have
\begin{equation*}
\begin{split}
\sum_{l_1, \cdots, l_{2m}} & \mathfrak{B}(v_{l_1}, v_{l_2})\cdots \mathfrak{B}(v_{l_{2m}}, v_{l_1})  \\
& =\sum_{l_1, l_{m+1}} \sum_{l_2, \cdots, l_m}\sum_{l_{2m}, \cdots, l_{m+2}} \mathfrak{B}(v_{l_1}, v_{l_2})\cdots \mathfrak{B}(v_{l_m}, v_{l_{m+1}})\mathfrak{B}(v_{l_{m+1}}, v_{l_{m+2}})\cdots \mathfrak{B}(v_{l_{2m}}, v_{l_1})  \\
& = \sum_{l_1, l_{m+1}} \big( \sum_{l_2, \cdots, l_m} \prod_{j=1}^m \mathfrak{B}(v_{l_j}, v_{l_{j+1}}) \big)^2 \\
& \geq 0,
\end{split}
\end{equation*}
since $\mathfrak{B}$ is symmetric and real-valued.
\end{proof}

We continue the proof of Theorem \ref{TheoremSection1ConverseInequality}. Recall the first line in \eqref{EquationSection3DirectInequality1}:
\begin{equation}  \label{EquationSection3ConverseInequality1}
\sum_{(l_j)_{j\in \left [2r \right]}}^* \Lambda(f_{l_1}, \cdots, f_{l_{2r}})=\sum_{\mathscr{P}} C_{\mathscr{P}} \sum_{(l_P)_{P\in \mathscr{P}}} \Lambda(f_{l_{\mathscr{P}(1)}}, \cdots, f_{l_{\mathscr{P}(2r)}}).    
\end{equation}
Classify all partitions into three types:
\begin{enumerate}
\item $\mathscr{P}$ is single if there exists $P\in \mathscr{P}$ such that $\#P=1$.
\item $\mathscr{P}$ is double if $\#P=2$ for any $P\in \mathscr{P}$.
\item $\mathscr{P}$ is joint if it's not single or double.
\end{enumerate}
Theorem \ref{TheoremSection2AlgebraicIdentity1} gives $C_{\mathscr{P}}=(-1)^{r}$ for any double partition $\mathscr{P}$. Hence we can further write \eqref{EquationSection3ConverseInequality1} as
\begin{equation}  \label{EquationSection3ConverseInequality2}
\begin{split}
\sum_{(l_j)_{j\in \left [2r \right]}}^* \Lambda(f_{l_1}, \cdots, f_{l_{2r}})= & (-1)^{r}\sum_{\mathscr{P}: \text{$\mathscr{P}$ is double}} \sum_{(l_P)_{P\in \mathscr{P}}} \Lambda (f_{l_{\mathscr{P}(1)}}, \cdots, f_{l_{\mathscr{P}(2r)}})+  \\
& \sum_{\mathscr{P}: \text{$\mathscr{P}$ is single}} C_{\mathscr{P}} \sum_{(l_P)_{P\in \mathscr{P}}} \Lambda (f_{l_{\mathscr{P}(1)}}, \cdots, f_{l_{\mathscr{P}(2r)}})+  \\
& \sum_{\mathscr{P}: \text{$\mathscr{P}$ is joint}} C_{\mathscr{P}} \sum_{(l_P)_{P\in \mathscr{P}}} \Lambda (f_{l_{\mathscr{P}(1)}}, \cdots, f_{l_{\mathscr{P}(2r)}}).
\end{split}
\end{equation}
Integrate both sides of \eqref{EquationSection3ConverseInequality2} on $X$. Equation \eqref{EquationSection3DirectInequality2} implies
\begin{equation}  \label{EquationSection3ConverseInequality3}
\begin{split}
\int_X \sum_{\mathscr{P}: \text{$\mathscr{P}$ is double}} \sum_{(l_P)_{P\in \mathscr{P}}} & \Lambda (f_{l_{\mathscr{P}(1)}}, \cdots, f_{l_{\mathscr{P}(2r)}}) d\mu \leq  \\
& \sum_{\mathscr{P}: \text{$\mathscr{P}$ is single}} |C_{\mathscr{P}}| \int_X |\sum_{(l_P)_{P\in \mathscr{P}}} \Lambda (f_{l_{\mathscr{P}(1)}}, \cdots, f_{l_{\mathscr{P}(2r)}})| d\mu+  \\
& \sum_{\mathscr{P}: \text{$\mathscr{P}$ is joint}} |C_{\mathscr{P}}| \int_X |\sum_{(l_P)_{P\in \mathscr{P}}} \Lambda (f_{l_{\mathscr{P}(1)}}, \cdots, f_{l_{\mathscr{P}(2r)}})| d\mu.
\end{split}
\end{equation}

For double partitions, we have
\begin{equation}  \label{EquationSection3ConverseInequality4}
\int_X \sum_{\mathscr{P}: \text{$\mathscr{P}$ is double}} \sum_{(l_P)_{P\in \mathscr{P}}} \Lambda (f_{l_{\mathscr{P}(1)}}, \cdots, f_{l_{\mathscr{P}(2r)}}) d\mu \geq \int_X \big( \sum_l \Lambda(f_{l})^{2} \big)^r d\mu
\end{equation}
since the term corresponding to $\mathscr{P}=\big{\{}\{1, 2\}, \{3, 4\}, \cdots \{2r-1, 2r\} \big{\}}$ gives RHS exactly and the term corresponding to each remaining double partition is non-negative by Lemma \ref{LemmaSection3Positivity}.

For a single partition $\mathscr{P}$, we use \eqref{EquationSection3DirectInequality5} to bound it.

For each joint partition $\mathscr{P}$, there exists $P_0\in \mathscr{P}$ such that $\#P_0\geq 3$. Successively applying \eqref{EquationSection3DirectInequality4}, monotonicity of $l^p$ norms, and H\"older's inequality gives
\begin{equation*}
\begin{split}
|\sum_{(l_P)_{P\in \mathscr{P}}} \Lambda (f_{l_{\mathscr{P}(1)}}, \cdots, f_{l_{\mathscr{P}(2r)}})| & \leq \prod_{P\in \mathscr{P}} \big(\sum_l \Lambda(f_{l})^{\#P} \big) \\
& \leq  \big(\sum_l \Lambda(f_{l})^{\#P_0} \big) \big(\sum_l \Lambda(f_{l})^{2} \big)^{\frac{2r-\#P_0}{2}}  \\
& \leq \big(\sum_l \Lambda(f_{l})^{2r} \big)^{\frac{\theta}{2r}\#P_0} \big(\sum_l \Lambda(f_{l})^{2} \big)^{\frac{1-\theta}{2}\#P_0} \big(\sum_l \Lambda(f_{l})^{2} \big)^{\frac{2r-\#P_0}{2}},
\end{split}
\end{equation*}
where $\theta$ satisfies
\begin{equation*}
\frac{1}{\#P_0}=\frac{\theta}{2r}+\frac{1-\theta}{2}. 
\end{equation*}
Note that $\#P_0\geq 3$ implies $\theta\geq \frac{1}{3}$, so $\theta\cdot\#P_0\geq 1$. Using the monotonicity of $l^p$ norms gives
\begin{equation*}
|\sum_{(l_P)_{P\in \mathscr{P}}} \Lambda (f_{l_{\mathscr{P}(1)}}, \cdots, f_{l_{\mathscr{P}(2r)}})|\leq \big(\sum_l \Lambda(f_{l})^{2r} \big)^{\frac{1}{2r}} \big(\sum_l \Lambda(f_{l})^{2} \big)^{\frac{2r-1}{2}}.
\end{equation*}
Then we can use \eqref{EquationSection1ConverseInequalityNecessaryCondition} to estimate 
\begin{equation}  \label{EquationSection3ConverseInequality5}
\begin{split}
\int_X |\sum_{(l_P)_{P\in \mathscr{P}}} \Lambda (f_{l_{\mathscr{P}(1)}}, \cdots, f_{l_{\mathscr{P}(2r)}})| d\mu & \leq \int_X \big(\sum_l \Lambda(f_{l})^{2r} \big)^{\frac{1}{2r}} \big(\sum_l \Lambda(f_{l})^{2} \big)^{\frac{2r-1}{2}} d\mu  \\
& \leq \Big( \int_X \big(\sum_l \Lambda(f_{l})^{2r} \big) d\mu \Big)^{\frac{1}{2r}} \Big( \int_X \big(\sum_l \Lambda(f_{l})^{2} \big)^r d\mu \Big)^{\frac{2r-1}{2r}}  \\
& \lesssim \Big(\int_X \Lambda(\sum_{l} f_l)^{2r} d\mu \Big)^{\frac{1}{2r}}\Big( \int_X \big(\sum_l \Lambda(f_{l})^{2} \big)^r d\mu \Big)^{\frac{2r-1}{2r}}.
\end{split}
\end{equation}

Plug \eqref{EquationSection3ConverseInequality4}, \eqref{EquationSection3DirectInequality5}, and \eqref{EquationSection3ConverseInequality5} into \eqref{EquationSection3ConverseInequality3}. Divide both sides by $\|\Lambda(\sum\limits_{l} f_l)\|_{2r}^{2r}$ and let
\begin{equation*}
t:=\frac{\|\big( \sum\limits_{l}\Lambda(f_l)^2 \big)^{\frac{1}{2}}\|_{2r}}{\|\Lambda(\sum\limits_{l} f_l)\|_{2r}}.    
\end{equation*}
We obtain
\begin{equation*}
t^{2r}\leq \widetilde{Q}(t)
\end{equation*}
for some polynomial $\widetilde{Q}$ such that $deg(\widetilde{Q})\leq 2r-1$, which implies $t\lesssim_r 1$. This completes the proof of Theorem \ref{TheoremSection1ConverseInequality}.


\section{Sharpness of the Conditions}  \label{AppendixA}
Let
\begin{equation*}
\mathcal{Z}_{\uppercase\expandafter{\romannumeral4}}:=\{(l_1, \cdots, l_{2r}): 1\leq l_j\leq L, l_1, \cdots, l_{2r} \text{ are all distinct.}\}
\end{equation*}
The main result in \cite{GPRYANewTypeOfSuperOrthogonality} and our theorems assume the vanishing condition \eqref{EquationSection1Superorthogonality} on $\mathcal{Z}_{\uppercase\expandafter{\romannumeral4}}$. Is there a type of superorthogonality that guarantees the direct/converse inequality, with a test set of index tuples that is strictly smaller than $\mathcal{Z}_{\uppercase\expandafter{\romannumeral4}}$? This question was introduced in \cite{GPRYANewTypeOfSuperOrthogonality}.

For the direct inequality, one can immediately improve $\mathcal{Z}_{\uppercase\expandafter{\romannumeral4}}$ to
\begin{equation}  \label{EquationAppendixAWeakerVanishingSet}
\{(l_1, \cdots, l_{2r}): 1\leq l_j\leq L, |l_j-l_{j^{'}}|>O(1), \forall j\neq j^{'}.\}
\end{equation}
by dividing $\{f_l\}_{1\leq l\leq L}$ into groups. This improved vanishing set is essentially the same as $\mathcal{Z}_{\uppercase\expandafter{\romannumeral4}}$. Does there exists a really smaller one? The following result means that the cardinality of such a set must be comparable with $\mathcal{Z}_{\uppercase\expandafter{\romannumeral4}}$.
\begin{proposition}  \label{PropositionAppendixDirect}
Let $(X, \mu)$ be a probability space. Let $r, L\in \mathbb{N}_+, \mathcal{Z}\subset \mathcal{Z}_{\uppercase\expandafter{\romannumeral4}}$ and $C>0$. If
\begin{equation*}
\|\sum_{l=1}^L f_l\|_{2r}\leq C\|(\sum_{l=1}^{L} |f_l|^2)^{\frac{1}{2}}\|_{2r}
\end{equation*}
holds for any family of real-valued functions $\{f_l\}_{1\leq l\leq L}$ such that
\begin{equation*}
\int_{X} f_{l_1}f_{l_2}\cdots f_{l_{2r-1}}f_{l_{2r}} d\mu=0, \text{whenever\ } (l_1, \cdots, l_{2r})\in \mathcal{Z},
\end{equation*}
then
\begin{equation*}
\#\mathcal{Z}\gtrsim_{r, C} \#\mathcal{Z}_{\uppercase\expandafter{\romannumeral4}}.    
\end{equation*}
\end{proposition}
\begin{proof}
Pick $m=[C^2]+1$. We can assume $2r\leq m\leq L$. For any $Z\subset \{l: 1\leq l\leq L\}$ such that $\#Z=m$, the intersection $Z^{2r}\cap \mathcal{Z}$ must be non-empty. Otherwise take $f_l\equiv 1$ if $l\in Z$ and $f_l\equiv 0$ if $l\notin Z$. Then
\begin{equation*}
\|\sum_{l=1}^L f_l\|_{2r}=m>Cm^{\frac{1}{2}}=C\|(\sum_{l=1}^{L} |f_l|^2)^{\frac{1}{2}}\|_{2r},
\end{equation*}
a contradiction. Now we have
\begin{equation*}
\binom{L}{m}\leq \sum_{Z: \#Z=m} \sum_{(l_1, \cdots, l_{2r})\in \mathcal{Z}} 1_{(l_1, \cdots, l_{2r})\in Z^{2r}}=\sum_{(l_1, \cdots, l_{2r})\in \mathcal{Z}} \sum_{Z: \#Z=m}  1_{(l_1, \cdots, l_{2r})\in Z^{2r}}=(\#\mathcal{Z})\binom{L-2r}{m-2r},
\end{equation*}
which implies
\begin{equation*}
\#\mathcal{Z}\gtrsim_{r, C} L^{2r}\sim_{r} \#\mathcal{Z}_{\uppercase\expandafter{\romannumeral4}}.
\end{equation*}
and finishes the proof.
\end{proof}
\begin{remark}
Proposition \ref{PropositionAppendixDirect} only gives a lower bound on $\#\mathcal{Z}$, which is far from determining the structure of such sets.
\end{remark}

The converse inequality turns out to be more sensitive. The counterexample given below shows that even the slightly weaker vanishing set \eqref{EquationAppendixAWeakerVanishingSet} is not sufficient. 
\begin{proposition}
Fix an integer $r\geq 2$. For any $\epsilon>0$, there exists a probability space $(X, \mu)$ and a family of real-valued functions $f_l, 1\leq l\leq 2L$, such that
\begin{equation} \label{EquationSection1WeakerSuperorthogonality}
\int_{X} f_{l_1}f_{l_2}\cdots f_{l_{2r-1}}f_{l_{2r}} d\mu=0, \text{whenever\ } l_1, \cdots, l_{2r} \text{\ satisfy\ } |l_j-l_{j^{'}}|>1,
\end{equation}
and
\begin{equation*}
\|(\sum_{l=1}^{2L} |f_l|^{2r})^{\frac{1}{2r}}\|_{2r}\lesssim \|\sum_{l=1}^{2L} f_l\|_{2r}\lesssim \epsilon \|(\sum_{l=1}^{2L} |f_l|^2)^{\frac{1}{2}}\|_{2r}.
\end{equation*}
\end{proposition}
\begin{proof}
Given $\epsilon>0$, pick $L>\epsilon^{-10}$. Take a probability space $(X, \mu)$ such that there exists an i.i.d. sequence of real-valued bounded random variables $\{g_l\}_{1\leq l\leq L}$ satisfying $\mathbf{E}g_1=0, \mathbf{E}|g_1|^2\neq 0$. Define
\begin{align*}
& f_{2l-1}=g_l,  \\
& f_{2l}=-(1-\epsilon)g_l,
\end{align*}
for $1\leq l\leq L$. It's easy to see $\mathbf{E}g_{l_1}\cdots g_{l_{2r}}=0$ whenever $l_1, \cdots, l_{2r}$ are all distinct, so \eqref{EquationSection1WeakerSuperorthogonality} holds.
Moreover, Theorem \ref{TheoremSection1DirectInequality} implies
\begin{equation*}
\|\sum_{l=1}^{2L} f_l\|_{2r}=\epsilon \|\sum_{l=1}^{L} g_l\|_{2r}
\lesssim_{r} \epsilon\|( \sum_{l=1}^L |g_l|^2 )^{\frac{1}{2}}\|_{2r}\leq \epsilon\|(\sum_{l=1}^{2L} |f_l|^2 )^{\frac{1}{2}}\|_{2r}.
\end{equation*}
On the other hand, we have
\begin{equation*}
\|\sum_{l=1}^{2L} f_l\|_{2r}=\epsilon \|\sum_{l=1}^{L} g_l\|_{2r}\geq \epsilon \|\sum_{l=1}^{L} g_l\|_2=\epsilon L^{\frac{1}{2}} \|g_1\|_2\gtrsim L^{\frac{1}{2r}} \|g_1\|_{2r}\sim \|( \sum_{l=1}^{2L} |f_l|^{2r} )^{\frac{1}{2r}}\|_{2r}
\end{equation*}
by H\"older's inequality. This completes the proof.
\end{proof}

\bibliographystyle{plain}  
\bibliography{ref}

\Address

\end{document}